\documentclass[12pt]{article}
\usepackage{amssymb}
\usepackage{amsmath}
\usepackage{theorem}
\usepackage{epsfig}
\usepackage{verbatim}
\usepackage{graphicx}
\usepackage{enumerate}
\usepackage{enumitem}
\usepackage{cancel}
\usepackage{mathtools}
\usepackage{hyperref}
\usepackage{datetime}

\usepackage[square,sort,comma,numbers]{natbib}
\usepackage{color}

\newlength{\hchng}
\newlength{\vchng}

\newcommand {\bc} {\begin{center}}
\newcommand {\ec} {\end{center}}

\theoremstyle{plain}

\newtheorem{thm}{Theorem}[section]
\newtheorem{lem}[thm]{Lemma}
\newtheorem{defn}[thm]{Definition}
\newtheorem{prop}[thm]{Proposition}

\allowdisplaybreaks[2]

\newenvironment{proof}[1]{\begin{trivlist} \item[] {\em Proof of #1:}}{\hfill $\Box$
                      \end{trivlist}}

\newcommand{\ud}{\,\mathrm{d}}
\newcommand{\la}{\lambda}

\newcommand{\R}{\mathbb{R}}

\newcommand{\mS}{\mathbb{S}}

\newcommand{\pa}{\partial}
\newcommand{\eps}{\epsilon}

\newcommand{\tn}{\textbf{n}}

\newcommand{\Ga}{\Gamma}

\title{The Friedland-Hayman inequality and Caffarelli's contraction theorem}

\date{}     
\author{Thomas Beck\thanks{The first author was supported in part by NSF Grant 2042654.} \and David Jerison\thanks{The second author was supported in part by NSF Grant 1500771, a Simons Fellowship, a Guggenhein Fellowship, and Simons Foundation Grant (601948, DJ).}}
\date{\today}                                         

\begin{document}
\maketitle

\vspace{-0.3in}

\begin{center}
{\em In memory of Jean Bourgain}
\end{center}

{\bf Abstract:} \ 
The Friedland-Hayman inequality is a sharp inequality 
concerning the growth rates of homogeneous, harmonic functions with Dirichlet
boundary conditions on complementary cones dividing Euclidean space into
two parts.  In this paper, we prove an analogous inequality in which one divides a convex cone into two parts, placing Neumann conditions on the boundary of the convex cone, and Dirichlet conditions  on the interface.  This analogous inequality was already proved by us jointly with Sarah Raynor.  Here we present a new  proof that permits us to characterize the case of equality.   
In keeping with the two-phase free boundary theory introduced by Alt, Caffarelli, and Friedman, such an improvement can be expected to yield further regularity in
free boundary problems.

\section{Introduction}

The Friedland-Hayman inequality \cite{FH} is a sharp inequality 
concerning the growth rates of homogeneous harmonic functions with Dirichlet
boundary conditions on complementary cones dividing Euclidean space into
two parts.  It plays a crucial role in the interior regularity theory
of two-phase free boundary problems, as developed by Alt, Caffarelli, and Friedman
\cite{ACF}.   In this paper, we prove an analogous inequality in which
one divides a convex cone into two parts, placing Neumann
conditions on the boundary of the convex cone, and Dirichlet conditions
on the interface.  This analogous inequality was already proved in \cite{BJR}
and leads to regularity of two-phase
free boundaries at points near a fixed boundary with Neumann conditions 
in convex domains.  Here we present a new  proof of independent interest
that leads, in addition, to the characterization of the case of equality.
In keeping with the theory introduced in \cite{ACF} (see also \cite{CaS}),
such an improvement should ultimately yield further regularity properties of the free boundary.

The Friedland-Hayman inequality can be stated as follows.
\begin{thm} \label{thm:Friedland-Hayman} \emph{\cite{FH}} \ Let $u_1$ and $u_2$ be non-negative, H\"older continuous 
functions defined on $\R^n$, with $u_1u_2\equiv0$, and harmonic where they are positive, that is,
$\Delta u_i(x) = 0$ whenever $u_i(x)>0$.  If $u_i$ is homogeneous of
degree $\alpha_i$, then 
\[
\alpha_1  + \alpha_2 \ge 2.
\]
\end{thm}
Equality holds if and only if the two functions are (up to constant multiples)
the positive and negative parts of a linear function.  In other words,
after rigid motion, 
\[
u_1(x) = c_1 x_1^+; \quad u_2(x) = c_2 x_1^-, \quad c_1 > 0 , \ c_2>0,
\]
with $x_1$ the first coordinate of $x$ and 
$u^\pm (x) : = \max (0, \pm u(x))$.  This was proved in \cite{ACF} (Lemma 6.6)  in dimension 2.  Also, in Remark 6.1 of that paper, the authors show that the characterization
of equality is valid in all dimensions,  provided one knows that the case of equality in a rearrangment theorem on the sphere due to Sperner \cite{Spe} is acheived only for rotationally symmetric caps. This
case of equality in Sperner's rearrangement theorem was subsequently proved by Brothers and Ziemer \cite{BZ}. 

To state our main theorem, consider a convex, open cone $\Ga$,
and two nonempty, disjoint, open, connected, conic subsets $\Ga_\pm$, 
that is,
\[
\bar \Ga_+ \cup \bar \Ga_- \subset \bar \Ga, \quad \Ga_+ \cap \Ga_- = \emptyset, \quad \Ga_\pm \neq \emptyset.
\]
Define the Neumann and Dirichlet portions of the boundary, $\gamma_N^\pm$ and 
$\gamma_D^\pm$,  by 
\[
\gamma_N^\pm = \pa \Ga \cap \pa \Ga_\pm, \quad \mbox{and} \quad \gamma_D^\pm : = \Ga \cap \pa \Ga_\pm, \quad \mbox{so that} \quad \pa \Ga_\pm = \gamma_N^\pm \cup \gamma_D^\pm.
\]
\begin{thm}  \label{thm:main} (Friedland-Hayman inequality for convex cones) Suppose that $\Ga\subset\R^n$ is an open,
convex cone containing two disjoint, open, connected cones $\Ga_\pm$ as above.
Let $u_\pm$ denote the unique (up to constant
multiples) positive, harmonic function 
on $\Ga_\pm$ that is  homogeneous of positive degree and
satisfies the mixed boundary conditions, $u_\pm = 0$ on $\gamma_D^\pm$ and $(\pa/\pa \nu) u_\pm = 0$  on $\gamma_N^\pm$ in the weak sense.   If the degree of $u_\pm$ is denoted $\alpha_\pm$, then
\[
\alpha_+ + \alpha_- \ge 2.
\]
Moreover, if equality holds, then after rotation, there is an open, convex
cone $\Ga' \subset \R^{n-1}$ such that
\[
\Ga = \R\times \Ga', \quad \Ga_\pm = \{\pm x_1 > 0\} \times \Ga', \quad \mbox{and} \quad u_\pm(x) = c_\pm x_1^\pm
\]
for some constants $c_\pm >0$. 
\end{thm}
The original Friedland-Hayman inequality is the case $\Ga  = \R^n$
with $\gamma_N^\pm  = \emptyset$.   If $\Ga$ is a half space, then 
the result also follows from the original Friedland-Hayman inequality by an
argument using reflection.   The proof of the inequality stated in Theorem \ref{thm:main} in our joint work with 
Sarah Raynor \cite{BJR} uses the L\'evy-Gromov isoperimetric
inequality on Ricci non-negative manifolds.  As we mentioned
earlier, the new proof here will permit us to characterize the case of equality.

The exponents $\alpha_\pm$ can be expressed in terms
of the lowest eigenvalues of a mixed boundary problem on the spherical
cross sections $\Ga_\pm \cap \mS^{n-1}$.  The relationship
is given by \eqref{eq:sphere-eigenvalue} below. 
Those eigenvalues have a variational characterization which is 
important to the applications to the free boundary regularity, that
is, to the monotonicity formulas of \cite{ACF} and \cite{BJR}.  Variational
characterizations are also important to this proof, as we shall see.

By rearrangement, the
Friedland-Hayman inequality reduces to a family of
one-dimensional problems parametrized by the dimension $n$.  
On the other hand, 
the inequality in dimension $n+1$ implies the inequality in dimension $n$,
as one can see by  considering the product of a cone with a line.
Thus the one-dimensional inequality gets harder to prove as $n$ increases.
Beckner, Kenig and Pipher \cite{BKP} gave a more conceptual proof of the 
original Friedland-Hayman theorem that relies on this rearrangement, but
circumvents the Brothers-Ziemer result.     They took the limit as $n$ tends
to infinity and identified and fully analyzed the problem one obtains
in the limit, an extremal eigenvalue problem on the real line with a gaussian weight.
The limiting ``infinite-dimensional"
problem dominates all the finite dimensional ones, and its 
extremal sets are the half lines $x_1>0$, $x_1 < 0$. See Section 12.2 in \cite{CaS} for the details of this approach.

Rearrangement cannot be used to solve the problem on convex domains.
But it is possible to take the dimension to infinity
by considering cones of
the form $\Ga_\pm \times \R^N$ as $N\to \infty$.  Using this device we will
prove the following key proposition.

\begin{prop} \label{prop:main}  If $\alpha_+$ is the characteristic exponent associated with $\Ga_+$ in
Theorem \ref{thm:main}, then 
\[
\alpha_+ \ge 
\inf_{f\in Y_+} \frac{\int_{\Ga_+}\left| \nabla f(x) \right|^2 e^{-|x|^2/2} \, dx}{\int_{\Ga_+} f(x)^2 e^{-|x|^2/2} \, dx} \, ,
\]
with $Y_+ = \{f \in C_0^\infty(\R^n): f(x) = 0 \ \mbox{for} \ x\in \gamma_D^+\}$.
Evidently, the same result holds with $+$ replaced by $-$. 
\end{prop}

Note that the weight in the variational expression on the right side of the inequality
in Proposition \ref{prop:main} is the restriction to $\Ga_+$ of 
\[
e^{-|x|^2/2} e^{-F(x)} \, dx, \quad F(x) = 
\begin{cases}  
0  & \ x\in \Ga \\
 \infty & \ x\in \R^n \setminus \Ga.
\end{cases}
\]
Because $\Ga$ is convex, $F$ is a generalized convex function, that is, a convex function allowing for the value $+\infty$.   
Put another way, $e^{-F}$ is a generalized log-concave function.

Because our measure is ``more log concave" than the gaussian,
we will be able to invoke a variant of Caffarelli's contraction theorem
for the Brenier optimal transport mapping.  Recall that Brenier's mapping
can be characterized as follows.

\begin{thm}[Brenier \cite{Br}] \label{thm:Brenier}
Let $\mu$, $\nu$ be positive measures on $\mathbb{R}^{n}$ and with finite second order moments.  Suppose also that $\mu$ is  absolutely continuous with respect to the Lebesgue measure. Then, there exists a convex function $\varphi:\mathbb{R}^{n}\to\mathbb{R}$ such that $T = \nabla \varphi:\mathbb{R}^{n}\to\R^{n}$ transports $\mu$ onto $c\, \nu$. That is,  $\mu(T^{-1}(E)) = c\, \nu(E)$, with normalizing constant
$c = \mu(\R^n)/\nu(\R^n)$.   
\end{thm}

The variant of Caffarelli's contraction theorem that we require can be stated 
as follows.
\begin{thm} \label{thm:Caf}  Let $K$ be a convex subset of $\R^n$ and let
$F$ be a convex function on $\R^n$. Set
\[
\mu = e^{-|x|^2/2}\, dx, \quad \nu  = e^{-F} 1_K \mu.
\]
Then the Brenier mapping $T = \nabla \varphi$ is 
a contraction:  $|T(x) - T(y)| \le |x-y|$.  Put another way,
$D^2\varphi$ has eigenvalues bounded above by $1$. 
\end{thm}
Caffarelli's theorem is the case $K = \R^n$.  The Friedland-Hayman inequality 
will follow from Proposition \ref{prop:main} and Theorem \ref{thm:Caf} with $K$ a convex cone and $F \equiv 0$ on $K$.

De Philippis and Figalli \cite{DPF} have characterized 
the case of equality in the eigenvalue formulation of Caffarelli's theorem.
This result will permit us to characterize the case of equality in the Friedland-Hayman
inequality, as stated in Theorem \ref{thm:main}. 
Their theorem is stated here in our generalized setting.  
\begin{thm} \label{thm:DPF}
 Let $\varphi$ be as in Theorem \ref{thm:Caf}, and let
\[
0\leq \la_1(D^2\varphi(x)) \leq \cdots \leq \la_{n}(D^2\varphi(x)) \leq 1
\]
be the eigenvalues of $D^2\varphi(x)$. If for some $m$, $1\leq m \leq n$, 
$\la_{n-m+1}(D^2\varphi(x)) = 1$ for all $x\in\R^n$, then,
after translation and rotation, there is a convex set $K'\subset \R^{n-m}$ and a convex function
$G$ on $K'$ such that 
\[
K  = \R^m \times K'; \quad \nu = e^{-|x'|^2/2 -|y'|^2/2 -G(y')} 1_{K'}(y') \, dx' dy'
\]
with $x = (x',y') \in \R^m \times \R^{n-m}$.
\end{thm}

Our second application is to the case of equality of a Poincar\'e-Wirtinger type inequality: For a given convex domain $K\subset \R^n$, define $\mu_1(K)$ by
\begin{align} \label{eqn:PW}
\mu_1(K) = \inf_f  \left\{\displaystyle{\frac{\int_{K}\left|\nabla f(x)\right|^2 e^{-|x|^2/2} \, d x }{\int_{K}|f(x)|^2 e^{-|x|^2/2} \, dx}}\text{ }:\text{ } \int_{K} f(x) e^{-|x|^2/2} \, dx = 0 \right\}. 
\end{align}
Here the infimum is taken over function for which the expressions are finite.
The value $\mu_1(K)$ can also be viewed as the first non-zero eigenvalue of
$$
\begin{cases}
\text{div}\left(e^{-|x|^2/2} \nabla u\right)  =  -\mu e^{-|x|^2/2} u  \text{ in } K, \\
\frac{\pa u}{\pa n}  =  0  \text{ on } \pa K. 
\end{cases}
$$
In \cite{BCHT}, Brandolini, Chiacchio, Henrot, and Trombetti show that if $K\subset \R^n$ has $C^2$-smooth boundary, then
\begin{align} \label{eqn:PW1}
\mu_1(K) \geq  \mu_1(\R^n) = \mu_1(\R) = 1.
\end{align}
In \cite{BCKT}, this inequality is shown to hold for any convex planar domain, and if $K\subset \R^2$ is contained in a strip, then equality in \eqref{eqn:PW1} holds precisely when $K$ is itself a strip. Here, we use Theorems \ref{thm:Caf} and \ref{thm:DPF} to prove the theorem in the non-smooth case and describe fully the case of equality.
\begin{thm} \label{thm:PW}
For any convex domain $K\subset \R^n$, and $\mu_1(K)$ as in \eqref{eqn:PW}, we have
\begin{align*}
\mu_1(K) \geq \mu_1(\R) = 1.
\end{align*}
Moreover, $($up to a rotation$)$ equality holds precisely when $K$ is of the form $K = \R\times K'$, for a convex domain $K' \subset \R^{n-1}$. 
\end{thm}
Let us make a few remarks about the existing literature.  
Our proof of Theorem \ref{thm:Caf} will follow Caffarelli's proof, exploiting the fact that $\varphi$ satisfies a Monge-Amp\`ere equation and that the second difference of $\varphi(x)$ is well-behaved as $|x|$ tends to infinity. Other proofs, variants, and extensions of this theorem have also been given in \cite{FGP}, \cite{KM1}, \cite{Ko1}, and \cite{Va1}.  The theorem of De Philippis and Figalli, Theorem 1.2 \cite{DPF}
 identifying the case of equality in Caffarelli's original theorem, has an alternative proof
due to Cheng and Zhou.  That proof involves 
the first non-zero eigenvalue of a Laplacian with drift on a complete smooth metric space with a lower bound on the Bakry-\'Emery Ricci curvature (see Theorem 2 in \cite{CZ}).

The rest of the paper is structured as follows. In Section \ref{sec:contraction}, we prove the version of Caffarelli's contraction theorem and the equality case, Theorems \ref{thm:Caf} and \ref{thm:DPF}.   We then show how our two applications follow from these theorems in Section \ref{sec:consequences}. We first prove the Poincar\'e-Wirtinger type inequality and case of equality, which is a direct consequence of Theorems \ref{thm:Caf} and \ref{thm:DPF}.  Our version of
the Friedland-Hayman inequality requires Proposition \ref{prop:main}.
This converts our problem to one about Gaussian eigenvalues of a convex domain to which Theorems \ref{thm:Caf} and \ref{thm:DPF} apply
 in $\R^{n}$.    We emphasize here that our argument is inspired by
 the argument of Beckner, Kenig and Pipher.   Moreover, because our method
 reduces the case of the convex cone to case of the entire Euclidean space, it
 depends on the original Friedland-Hayman theorem and does not
 replace it.  We end by discussing a possible, natural variant
 of Caffarelli's contraction theorem for geodesically convex subsets
 of spheres.  This variant would provide an alternative
path to the main theorem, Theorem \ref{thm:main}.

\section{Proof of the Caffarelli contraction theorem} \label{sec:contraction}

Recall from Theorem \ref{thm:Brenier}, that $T = \nabla\varphi$ transports the gaussian measure $\mu = e^{-|x|^2/2}\, dx$ onto the measure $c\,\nu$, where $\nu = e^{-F(x)}1_K\, \mu$ for a convex domain $K\subset \R^n$ and a convex function $F$ on $K$. Here $\varphi(x)$ is a convex function on $\R^n$ and the constant $c$ is chosen so that $\mu$ and $c\,\nu$ have the same total measure. To prove Theorem \ref{thm:Caf}, we need to show that the eigenvalues of $D^2\varphi(x)$ are bounded above by $1$.
\begin{proof}{Theorem \ref{thm:Caf}}
We follow the proof Caffarelli used to prove Theorem 11 in \cite{Ca1}. In particular, rather than studying $D^2\varphi(x)$ directly, we work with second differences of $\varphi(x)$. That is, we fix $h>0$, and for each unit direction $e$, $x\in \R^n$, we set
 \begin{align*}
\delta_{e} \varphi(x) = \varphi(x + he) + \varphi(x -he) - 2 \varphi(x).
\end{align*}
To prove the theorem, we need to show that
\begin{align} \label{eqn:Caf-1}
0\leq \delta_{e} \varphi(x) \leq h^2
\end{align}
for all $x$, $e$, and $h>0$, since then letting $h\to0$ gives the desired upper bound. To prove \eqref{eqn:Caf-1} it is sufficient to study $\delta_{e} \varphi(x)$ at a point where it achieves its maximum in both $x$ and $e$, together with its behavior as $|x|$ tends to infinity. However, in the case where the convex set $K$ is not smooth, strictly convex, and bounded, then the behavior of $\varphi(x)$ at infinity can be more complicated. Therefore, we form a sequence of smooth, strictly convex, and bounded sets $K_j\subset K$, which converges to $K$ in Hausdorff distance on compact sets  \cite{Sc1}. (By strictly convex, we mean that each tangent plane to $K_j$ touches $\pa K_j$ at a unique point.) We also obtain corresponding Brenier maps $T_j = \nabla \varphi_j$ transporting $\mu = e^{-|x|^2/2}\, dx$ onto $c_j\nu_j$, with $\nu_j = e^{-F(x)}1_{K_j}\, \mu$. Here the constant $c_j>0$ is chosen to ensure that $c_j\nu_j$ and $\mu$ have the same total measure.

As $F$ is a convex function on $\R^n$, it is in particular in $L^{\infty}(K\cap B_R)$ for all $R>0$. Since the sets $K$ and $K_j$ are convex, this ensures that $\varphi$ and $\varphi_j$ are $C^2$ (see Theorem 1 in \cite{CEF}). Therefore, $\varphi_j$ is a classical solution to the Monge-Amp\'ere equation 
\begin{align} \label{eqn:MA}
\text{det}\left(D^2\varphi_j(x)\right) = \frac{e^{-|x|^2/2}}{\exp\left\{-|\nabla\varphi_j(x)|^2/2-F(\nabla \varphi_j(x))\right\}},
\end{align}
and $\varphi$ satisfies the analogous equation. Moreover, since $K_j$ converges to $K$ on compact sets, the measures $e^{-F(x)}1_{K_j}\, \mu$ converge strongly to $\nu$ as $j$ tends to infinity. Therefore, $\varphi_j(x)$ converges to $\varphi(x)$ uniformly on compact sets (see Theorem 3 in \cite{Ca1}, also \cite{Vi1}, 5.23).  To show \eqref{eqn:Caf-1} and complete the proof of the theorem, it is thus sufficient to establish
\begin{align} \label{eqn:Caf-2}
0\leq \delta_{e} \varphi_j(x) \leq h^2
\end{align}
for all $x$, $e$, and $h>0$, and $j$ fixed.

Note that the lower bound is guaranteed since $\varphi_j$ is convex. To prove the upper bound, suppose first that $\delta_e\varphi_j(x)$ attains its maximum in $x$ and $e$ at $x=x_0$ and $e=e_0$. Then, as shown in the proof of Theorem 11 in \cite{Ca1} (see also \cite{Ca2}), using the fact that $\varphi_j(x)$ satisfies the Monge-Amp\`ere equation in \eqref{eqn:MA}, it satisfies a maximum principle ensuring that $\delta_{e_0}\varphi_j(x_0) \leq h^2$.

To complete the proof we therefore need to study the behavior of $\delta_e\varphi_j(x)$ as $|x|$ tends to infinity. This part of the proof is the reason for using the smooth, strictly convex approximating sets $K_j$, and also why we work with the second difference rather than the second derivative directly. Since the sets $K_j$ will not be balls centered at the origin when $K$ is a proper subset of $\R^n$, this part of the proof requires a small modification of Caffarelli's proof, and so we write it out in detail. 
\begin{lem} \label{lem:infinity}
Suppose that $|x|$ tends to infinity, with $\tfrac{x}{|x|}$ converging to a direction ${\tn}$. Then, 
\begin{align*}
\nabla\varphi_j(x) \to  {y}_{{\tn}},
\end{align*}
with uniform convergence in the direction $\tn$. Here ${y}_{{\tn}}$ is the unique point on $\pa K_j$ with outward unit normal pointing in the direction ${\tn}$.
\end{lem}
\begin{proof}{Lemma \ref{lem:infinity}}
The uniqueness of the point $y_{\tn}$ follows immediately from the strict convexity and smoothness of $K_j$. Given ${x}\in\mathbb{R}^{n}$, let ${y} = \nabla \varphi_j ({x}) \in K_j$. Define $\Gamma_{y}(\theta)$ to be the cone of vertex ${y}$ pointing in the direction of ${x}$, with angle $\theta$ (here $\theta$ is fixed, with $0<\theta<\tfrac{\pi}{2}$), so that
\begin{align*}
\Gamma_{{y}}(\theta) = \{{y}'\in\mathbb{R}^{n}: \text{angle}(x,{y}'-{y}) \leq \theta\}.
\end{align*}
By the cyclical monotonicity of the optimal transport mapping, if ${y}' = \nabla \varphi_j({x}')$, then the inner product of ${x}'-{x}$ and ${y}' -{y}$ is nonnegative. Therefore, the pre-image of $\Gamma_{y}(\theta)$ under $\nabla \varphi_j$ is contained in the cone
\begin{align*}
\Gamma_{x}(\theta) = \{{x}'\in\mathbb{R}^{n}: \text{angle}(x,{x}'-{x}) \leq \theta + \tfrac{\pi}{2}\}.
\end{align*}
There exists a constant $a_{\theta}>0$ such that the complement of $\Gamma_{x}(\theta)$ contains the ball of radius $a_{\theta}|x|$ centred at the origin. In particular, as $|x|$ tends to infinity, for $\theta$ fixed, the Gaussian measure of $\Gamma_{x}(\theta)$ tends to zero. Since the map $\nabla \varphi_j$ is measure preserving, and the measure $\nu_j$ is bounded from below on $K_j$, this means that the Lebesgue measure of $\Gamma_{{y}}(\theta)\cap K_j$ tends to $0$. Therefore, given $\eta>0$ and $0 < \theta < \tfrac{\pi}{2}$, there exists $M_{\theta}>0$ such that if $|x|>M_{\theta}$, then the Lebesgue measure of $\Gamma_{{y}}(\theta)\cap K_j$ is less than $\eta$. This in particular ensures dist$(y,\pa K_j)<C_1\eta$, where $C_1$ is a constant depending only on the Lipschitz bound for the convex set $K_j$. As $\theta$ tends to $\tfrac{\pi}{2}$, the cone $\Gamma_y(\theta)$ approaches the half-plane passing through $y$ in the direction $\tfrac{x}{|x|}$, and this forces $y$ to approach $y_n$, the unique point on $\pa K_j$ with outward unit normal $\textbf{n}$. More precisely, for any $\eta>0$, if $|\textbf{n}-\textbf{m}| <\eta$, then $|y_{\textbf{n}}-y_{\textbf{m}}| < C_2\eta$, for a constant $C_2$ depending only on the strict convexity of $K_j$ (but not $\textbf{n}$). Therefore, given $\eta>0$, we can choose $\delta>0$ and $\theta^* < \tfrac{\pi}{2}$ such that if $\left|\tfrac{x}{|x|} - \textbf{n}\right| < \delta$, with $|x|>M_{\theta^*}$, then $|y-{y}_\textbf{n}| < \eps$. This proves that $y$ converges to $y_{\textbf{n}}$, uniformly in the direction $\textbf{n}$. 
\end{proof}
To conclude the proof of the theorem, we note that
\begin{align*}
0\leq h^{-2}\delta_e\varphi_j(x) \leq \frac{\nabla \varphi_j({x} + h{e})\cdot {e} - \nabla \varphi_j({x}-h{e})\cdot {e}}{2h},
\end{align*}
and for $h>0$, ${e}$ fixed, $\left|\tfrac{{x} + h{e}}{|{x} + h{e}|} - \tfrac{{x} - h{e}}{|{x} - h{e}|}\right|$ tends to $0$ as $|{x}|$ tends to infinity. Therefore, Lemma \ref{lem:infinity} implies that $\delta_e\varphi_j(x)$ tends to $0$ as $|x|$ tends to infinity, and  hence $\delta_e\varphi_j(x)\leq h^2$ for all ${x}$, ${e}$ and $h>0$ as required.
\end{proof}
\begin{proof}{Theorem \ref{thm:DPF}}
Now that we have proved Theorem \ref{thm:Caf}, after ordering the eigenvalues of $D^2\varphi(x)$, we can ensure that
\begin{align*}
0\leq \la_1(D^2\varphi(x)) \leq \cdots \leq \la_{n}(D^2\varphi(x)) \leq 1.
\end{align*}
To deal with the case of equality, where for some $k\geq1$, $\la_{n-k+1}(D^2\varphi(x)) = 1$ for all $x\in\R^n$, the proof of De Philippis and Figalli in Theorem 1.2 in \cite{DPF} (used to deal with the case of stability in Caffarelli's original contraction theorem) still applies and so we just briefly summarize their proof:  Defining the convex function $\Psi(x) = \tfrac{1}{2}|x|^2 - \varphi(x)$, the assumption of the theorem implies that det$(D^2\Psi)(x) = 0$. Subtracting a linear function from $\Psi(x)$ (which only translates $\nu$), we can assume that $\Psi(x) \geq\Psi(0) = 0$. Combining det$(D^2\Psi)(x) = 0$ with an Alexandrov estimate, de Philippis and Figalli show that the set $\Sigma = \{\Psi = 0\}$ does not have an exposed point. Therefore, the set $\Sigma$ must contain a line (see \cite{Bar}, Lemma 3.5 in Chapter 2). Rotating so that this line is $\mathbb{R}{e}_1$, and combining this with the convexity of $\Psi$ implies that $\pa_{{e}_1}\Psi({x}) \equiv 0$. Therefore, we can view $\Psi$ as a function of the variables $x' = (x_2,\ldots,x_n)$, and write the mapping $T = \nabla \varphi$ as $T(x) = \left(x_1,x' - \nabla \Psi(x')\right)$. This means that the measure $\nu$ can be written as $\nu = e^{-x_1^2/2}\, dx_1\otimes \nu_1$, where $T_1(x') = x' - \nabla \Psi(x')$ transports the $(n-1)$-dimensional gaussian measure $e^{-|x'|^2/2}$ onto $c\nu_1$.  
Moreover, we can write $\nu_1$ as $e^{-|x'|^2/2-G_1(x')}1_{K_1} \, dx'$, for a convex set $K_1$ and convex function $G_1$, since these properties are preserved under taking marginals (see Theorem 4.3 in \cite{BL}). This proves the theorem for $k=1$, and by recursively applying this argument the theorem holds.
\end{proof}

\section{Consequences of the contraction theorem} \label{sec:consequences}

In this section we use Theorems \ref{thm:Caf} and \ref{thm:DPF} to study the Friedland-Hayman and Poincar\'e-Wirtinger type inequalities discussed in the introduction. The Poincar\'e-Wirtinger result follows as a direct application, 
so we will prove it first.
\subsection{A Poincar\'e-Wirtinger inequality} 
\begin{proof}{Theorem \ref{thm:PW}}
Let $u(x)$ be the eigenfunction corresponding to minimizing the quantity $\mu_1(K)$ given in \eqref{eqn:PW}.  Applying Theorem \ref{thm:Caf} with $F(x) \equiv 0$ (so that $\nu = e^{-|x|^2/2}1_K\, dx$), we obtain a transport map $T = \nabla \varphi$ from $\mu = e^{-|x|^2/2} \,dx$ to $c\,\nu$, such that $T$ is Lipschitz, with Lipschitz constant bounded by $1$. We define $v(x)$ by $v(x) = u(T(x))$, and since $T$ is a transport map, we have
\begin{align*}
\int_{\R^n} v(x) e^{-|x|^2/2} \, dx = \int_{\R^n}u(T(x)) e^{-|x|^2/2} \, dx = c\int_{K}u(x) \, d\nu = 0. 
\end{align*}
Therefore, $v(x)$ is an admissible test function for $\mu_1(\R^n)$, and so
\begin{align} \label{eqn:PWproof1}
\mu_1(\R) = \mu_1(\R^n) \leq \displaystyle{\frac{\int_{\R^n}\left|\nabla v(x)\right|^2e^{-|x|^2/2} \, dx }{\int_{\R^n}|v(x)|^2e^{-|x|^2/2}\, dx}}.
\end{align} 
Moreover, since the Lipschitz constant of $T$ is bounded by $1$, we have
\begin{align} \label{eqn:PWproof2}
\left|\nabla v(x)\right| \leq \left|(\nabla u)(T(x))\right|
\end{align}
for all $x\in \mathbb{R}^n$. Combining this with the fact that $T$ is a transport map, we can use \eqref{eqn:PWproof1} to obtain
\begin{align*}
\mu_1(\R) \leq \displaystyle{\frac{\int_{K}\left|\nabla u(x)\right|^2e^{-|x|^2/2} \, dx }{\int_{K}|u(x)|^2e^{-|x|^2/2}\, dx}} = \mu_1(K).
\end{align*} 
To deal with the equality case $\mu_1(K) = \mu_1(\R) = 1$ we use Theorem \ref{thm:DPF}. If $K$ contains a line (say $\R_{e_1}$), then we can set $u(x) = x_1$ to obtain $\mu_1(K) = 1$. Now suppose that  $K$ does not contain a line and again let $u(x)$ be the eigenfunction corresponding to $\mu_1(K)$. Then by Theorem \ref{thm:DPF}, there exist $\eps>0$ and a set $U\subset \R^n$ of positive measure for which $\la_j(D^2\varphi(x))\leq1-\eps$ for all $x\in U$, $1\leq j \leq n$. The image of $U$ under $T$ also has positive measure. Setting $v(x) = u(T(x))$, we thus have
\begin{align*}
\int_{U}\left|\nabla v(x)\right|^2e^{-|x|^2/2} \, dx  \leq (1-\eps) \int_{U}\left|\nabla u(T(x))\right|^2e^{-|x|^2/2} \, dx.
\end{align*}
Using this inequality in the above argument in place of \eqref{eqn:PWproof2}, we obtain $\mu_1(K)>1$ and so we cannot have equality unless $K$ contains a line. 
\end{proof}

\subsection{The Friedland-Hayman inequality}

In this section we prove Theorem \ref{thm:main}.  Recall that $u_+$ is
a harmonic function on the cone $\Gamma_+$ of homogeneous 
degree $\alpha_+$ satisfying
boundary conditions $u_+ = 0$ on $\gamma_D^+$ and 
$(\partial/\partial\nu) u_+ = 0$ (in the weak sense)
on $\gamma^+_N$. The first step is to prove Proposition \ref{prop:main}, a lower bound 
on the characteristic exponent $\alpha_+$ in 
 terms of the lowest Gaussian eigenvalue on the cone with the same boundary
conditions.  The convexity of the cone $\Gamma$ is not used in this proof.

Fix an integer $N\ge 0$, and extend the function $u_+$ to
the product cone $\Gamma_+ \times \R^N$ to be constant in the extra variable:
\[
u_N(x,y) = u_+(x), \quad (x,y) \in \Gamma_+\times \R^N.
\]
Denote the spherical cross-section of the cone $\Gamma_+ \times \R^N\subset \R^{m}$ by 
\[
\Sigma_m = \{(x,y) \in \Gamma_+\times \R^N: |x|^2 + |y|^2 = 1\},
\]
where $m = n+N$. Let $d\sigma_m$, $\nabla_m$, $\Delta_m$ denote the spherical
measure, gradient and Laplace-Beltrami operator on the unit $(m-1)$ sphere in
$\R^m$.
Then, setting $\alpha = \alpha_+$, $u_N$ is homogeneous of degree
$\alpha$ in the product cone in $\R^m$, and so by separation of variables,
\[
\Delta_m u_N = -\alpha(\alpha+ m-2) u_N \quad \mbox{on}
\quad  \Sigma_m. 
\]
Integrating by parts, using the boundary conditions, we get
\begin{equation}\label{eq:sphere-eigenvalue}
\int_{\Sigma_m} |\nabla_m u_N|^2 \, d\sigma_m = \alpha(\alpha+m-2)\int_{\Sigma_m} u_N^2 \, d\sigma_m \ .
\end{equation}
We rewrite this as 
\[
\alpha(\alpha+ N+ n-2) = 
\frac{\int_{\Sigma_m} |\nabla_m u_N|^2 \, d\sigma_m}
{\int_{\Sigma_m} u_N^2 \, d\sigma_m} 
\]
The integrand in the numerator, on $|x|^2 + |y|^2=1$, is given by
\[
|\nabla_m u_N|^2 
= |\nabla_{x,y} u_N (x,y)|^2 - |(x,y)\cdot \nabla_{x,y}u_N(x,y)|^2 
= |\nabla  u_+(x))|^2 - |x\cdot \nabla u_+(x)|^2 =
|\nabla  u_+(x)|^2 - (\alpha \, u_+(x))^2,
\]
and the formula for the integral of a function $f$ on $\Sigma_m$ that depends 
only on $x$ is  
\[
\int_{\Sigma_m} f(x) \, d\sigma_m(x,y) = c_N \int_{\Gamma_+ \cap \{|x|<1\}} f(x) (1-|x|^2)^{N/2 - 1} dx.
\]
Therefore, 
\[
\frac{\alpha(\alpha+ N+n-2)}{N} = 
\frac{\frac1N  \int_{\Gamma_+ \cap \{|x|<1\}} 
[|\nabla  u_+(x))|^2 - (\alpha u_+(x))^2] (1-|x|^2)^{N/2 - 1} dx}
 {\int_{\Gamma_+ \cap \{|x|<1\}} u_+(x)^2 (1-|x|^2)^{N/2 - 1} dx} \ .
 \]
Notice that the left hand side is $\alpha + O(1/N)$ as $N\to \infty$.

Next, change variables by setting $f_N(z) = u_+(z/\sqrt{N})$, $|z| < \sqrt{N}$,
to obtain
\begin{align*}
\alpha \
& = \
\frac{\int_{\Gamma_+ \cap \{|z|<\sqrt{N}\}} 
[|\nabla  f_N(z)|^2 - \frac1N |\alpha \, f_N(z)|^2] \left(1-\frac{|z|^2}N\right)^{N/2 - 1} dz}
{\int_{\Gamma_+ \cap \{|z|<\sqrt{N}\}} 
f_N(z)^2 \left(1-\frac{|z|^2}N\right)^{N/2 - 1} 
dz} \ + \ O(1/N) \\
& = \
\frac{\int_{\Gamma_+ \cap \{|z|<\sqrt{N}\}} 
|\nabla  f_N(z)|^2 \left(1-\frac{|z|^2}N\right)^{N/2 - 1} dz}
{\int_{\Gamma_+ \cap \{|z|<\sqrt{N}\}} 
f_N(z)^2 \left(1-\frac{|z|^2}N\right)^{N/2 - 1} 
dz} \ + \ O(1/N)
\end{align*}
Because
\[
\left(1-\frac{|z|^2}N\right)^{N/2 - 1}  \to e^{-|z|^2/2} \ \mbox{as} \ N\to \infty,
\]
this nearly completes the proof.  The additional property we need
to check is that as $N \to \infty$, 
\begin{equation}\label{eq:cut-off}
\frac{\int_{\Gamma_+ \cap \{\sqrt{N} - 1 <|z|<\sqrt{N}\}} 
[f_N(z)^2 + |\nabla f_N(z)|^2 ] \left(1-\frac{|z|^2}N\right)^{N/2 - 1} dz}
{\int_{\Gamma_+ \cap \{|z|<\sqrt{N}\}}f_N(z)^2 \left(1-\frac{|z|^2}N\right)^{N/2 - 1} dz} 
\to 0. 
\end{equation}
An estimate like \eqref{eq:cut-off} is required because $f_N$ is not a suitable
test function.  Although $f_N$ does vanish on $\gamma^+_D$ as required, it
does not vanish on the outer boundary $|z| = \sqrt{N}$. 
The simplest truncation is by a radial cut-off function of slope $1$ 
on a band of unit width 
$\sqrt{N} - 1 < |z| < \sqrt{N}$, which gives a legitimate test function and proves our proposition, assuming \eqref{eq:cut-off} holds. 

Finally, \eqref{eq:cut-off}  follows from the fact that $f_N$ and $\nabla f_N$ 
grow like powers of $|z|$, whereas the weight resembles $e^{-|z|^2/2}$.   In detail, set 
\[
A = \int_{\Sigma_n} u_+^2 \, d\sigma_n .
\]
Then, since $u_+$ is homogeneous of degree $\alpha$, 
\[
\int_{\Gamma_+ \cap \{|z|<\sqrt{N}\}}f_N(z)^2 w(|z|)\, dz
=  A \int_0^{\sqrt{N}} \frac{r^{2\alpha}}{N^\alpha} w(r) r^{n-1} \, dr
\]
By equation \eqref{eq:sphere-eigenvalue} for $m=n$, $N=0$, we have
\[
\int_{\Sigma_n} |\nabla u_+|^2 \, d\sigma_n = \alpha(\alpha  + n-2) A.
\]
Thus,  since $\nabla u_+$ is homogeneous of degree $\alpha-1$, 
\[
\int_{\Gamma_+ \cap \{|z|<\sqrt{N}\}}  |\nabla f_N(z)|^2   w(|z|) dz 
= 
\alpha(\alpha  + n-2) A \int_0^{\sqrt{N}} \frac{r^{2\alpha-2}}{N^\alpha} w(r) r^{n-1} \, dr.
\]
Note that the denominator $N^\alpha$ is the same in both formulas because
of the choice of change of variable $z/\sqrt{N}$.  

With these formulas for the numerator and denominator, 
one sees that \eqref{eq:cut-off} is valid.  Indeed, setting $w(r) = \left(1-r^2/N\right)^{N/2-1}$,
\[
\int_{\sqrt{N}-1}^{\sqrt{N}} ( r^{2\alpha} + r^{2\alpha-2}) (1-r^2/N)^{N/2-1} r^{n-1}\, dr
\]
tends to zero very fast since $(1-r^2/N)^{N/2-1} \approx e^{-r^2/2} \approx e^{-N/2}$
in the range $\sqrt{N} - 1 < r < \sqrt{N}$, whereas
\[
\int_{0}^{\sqrt{N}} r^{2\alpha}  (1-r^2/N)^{N/2-1} r^{n-1}\, dr
\]
tends to a positive limit.  This concludes the proof of Proposition \ref{prop:main}.

We are now ready to prove the Friedland-Hayman inequality in the
original case with $\Gamma = \mathbb{R}^{n}$.
\begin{defn} \label{def:full}
For $\Omega\subset\R^{n}$ open, with Gaussian measure satisfying $\int_{\Omega}e^{-|x|^2/2}\ud x <1$, define $\lambda_g(\Omega)$ by
\begin{align*}
\lambda_g(\Omega) =  \inf_{w\in H_0^1(\Omega)} \frac{\int_{\Omega}\left| \nabla w(x) \right|^2 e^{-|x|^2/2} \, dx}{\int_{\Omega} w(x)^2 e^{-|x|^2/2} \, dx}.
\end{align*}
\end{defn}
Alternatively, $\la_g(\Omega)$ is the first eigenvalue of
$$
\begin{array}{rll}
-(\partial/\partial x_i) \left(e^{-|x|^2/2} \partial u/\partial x_i\right)  = & 
\lambda \, e^{-|x|^2/2}u(x)  & \emph{ in } \Omega \\
u(x)  = & 0 &  \emph{ on } \pa\Omega.
\end{array}
$$
Via Gaussian symmetrization, the following theorem holds:
\begin{thm}[Erhard, Proposition 2.3 in \cite{Er} and Carlen, Kerce, Theorem 3 in \cite{CK}] \label{thm:BCF}
Let $\Omega$ and $\lambda_g$ be as in \emph{Definition \ref{def:full}}. Let $H$ be a half-space with the same Gaussian measure as $\Omega$, then
\begin{align*}
\la_g(\Omega) \geq \la_g(H),
\end{align*}
with equality if and only if $\Omega$ is equal to $H$, up to a rotation.
\end{thm}
For half spaces, the corresponding eigenfunction is a function of a single variable. Therefore, we can apply the results of Beckner, Kenig, and Pipher.
\begin{thm}[Beckner, Kenig, Pipher \cite{BKP}; in Section 12.2 of \cite{CaS}]
\label{thm:BKP}
Let $H_\pm$ be complementary half-spaces. Then,
\begin{align*}
\la_g(H_+) + \la_g(H_-) \geq 2,
\end{align*}
with equality if and only if $\int_{H_+}e^{-|x|^2/2} \, dx = \int_{H_-}e^{-|x|^2/2} \, dx$.
\end{thm}
In particular, combining these two theorems with Proposition \ref{prop:main} establishes the original Friedland-Hayman inequality and classifies the case of equality. For the general case where $\Gamma$ is a proper, convex subset of $\R^n$, we will use Theorems \ref{thm:Caf} and \ref{thm:DPF} to compare the Gaussian eigenvalues on $\Gamma_{\pm}$ to those on complementary subsets of $\R^{n}$. In particular, the proof is now similar to that of Theorem \ref{thm:PW} with some small modifications coming from the fact that we have to consider eigenvalues on complementary subsets. 

We first apply Theorem \ref{thm:Caf} with $K = \Gamma$ and $F(x)\equiv0$. This gives a transport map $T = \nabla \varphi$ from $\mu = e^{-|x|^2/2}\, dx$ to $c_{n,K}\nu$, with $\nu = e^{-|x|^2/2}1_\Gamma \, dx$ (and $c_{n,K}$ a constant ensuring $c_{n,K}\nu$ and $\mu$ have the same total measure), such that $T$ is Lipschitz with Lipschitz constant bounded above by $1$. Let $w_{\pm}$, defined on $\Gamma_{\pm}$ respectively, be admissible test functions for the infimum given in Proposition \ref{prop:main}. Extending $w_{\pm}$ by zero so that they are defined on $\Gamma$, we set $v_+(x) = w_+(T(x))$, $v_-(x) = w_-(T(x))$. Since $T$ is a transport map, we have
\begin{align} \label{eqn:transport1}
\int_{\R^n} v_{\pm}(x)^2  e^{-|x|^2/2}\, dx = \int_{\R^n}w_{\pm}(T(x))^2 e^{-|x|^2/2} \, dx = c\int_{\Gamma} w_i(x)^2 \,d \nu.
\end{align}
$T$ is Lipschitz, with Lipschitz constant bounded by $1$, and so
\begin{align} \label{eqn:transport2}
\left|\nabla v_{\pm}(x)\right| \leq \left|\left(\nabla w_{\pm}\right)(T(x))\right| \left|\nabla T(x)\right| \leq  \left|\left(\nabla w_{\pm}\right)(T(x))\right|.
\end{align}
Combining \eqref{eqn:transport1} and \eqref{eqn:transport2} gives
\begin{align} \label{eqn:transport3}
\frac{\int_{\Gamma_{+}}\left| \nabla w_+(x) \right|^2 e^{-|x|^2/2} \, dx}{\int_{\Gamma_{+}} w_+(x)^2  e^{-|x|^2/2} \, dx} \geq \frac{\int_{\R^{n}}\left| \nabla v_+(x) \right|^2  e^{-|x|^2/2} \, dx}{\int_{\R^{n}} v_+(x)^2 e^{-|x|^2/2} \ud x},
\end{align}
and the same for $\Gamma_{-}$, $w_-$ and $v_-$. Since $w_1$ and $w_2$ have disjoint supports in the cone $\Gamma$, so do the functions $v_1$ and $v_2$ in $\R^{n}$. Applying Theorems \ref{thm:BCF} and \ref{thm:BKP} thus implies that 
\begin{align*}
\frac{\int_{\R^{n}}\left| \nabla v_+(x) \right|^2  e^{-|x|^2/2} \, dx}{\int_{\R^{n}} v_+(x)^2 e^{-|x|^2/2} \, dx} + \frac{\int_{\R^{n}}\left| \nabla v_-(x) \right|^2  e^{-|x|^2/2} \ud x}{\int_{\R^{n}} v_-(x)^2  e^{-|x|^2/2} \, dx} \geq 2,
\end{align*}
and by Proposition \ref{prop:main}, this proves the desired inequality in Theorem \ref{thm:main}.

We now turn to the case of equality. Suppose that 
the cone $\Gamma$ does not contain a line.  Then by Theorem \ref{thm:DPF}, there exist $\eps>0$ and a set $U_1\subset\R^{n}$ of positive measure for which $\la_{j}(D^2\varphi(x))\leq1-\eps$ for all $x\in U_1$, $1\leq j \leq n$. The image of $U_1$ under $T$ is also of positive measure, and so without loss of generality, assume that $T(U_1)\cap \Gamma_{+}$ has positive measure. Letting $w_+$ be the minimizer for the infimum in Proposition \ref{prop:main} for $\Gamma_{+}$, we therefore have
\begin{align*}
\left|\nabla v_+(x)\right| \leq (1-\eps) \left|\left(\nabla w_+\right)(T(x))\right| \quad \text{on } T^{-1}\left(T(U_1)\cap \Gamma^{+}\right).
\end{align*}
Inserting this strict inequality into the argument above ensures that we cannot have the equality $\alpha_{+}+\alpha_{-} = 2$.

Now suppose that we have equality. Then we can write the cone $\Gamma$ in the form $\Gamma = \R^{k}\times \Gamma'$, where $\Gamma'$ does not contain a line. In particular, by Theorem \ref{thm:DPF}, we have $T(x) = (x_1-p_1,\ldots,x_k-p_k,T'(x_{k+1},\ldots,x_{n+1}))$ for some fixed $(p_1,\ldots,p_k)\in\R^k$ and with  $|\nabla T'|<1$ on a set of positive measure $U_2$ in $\R^{n-k}$. Assuming that $T'(U_2)\cap \Pi'_{n-k}\Gamma_{+}$ has positive measure in $\R^{n-k}$ (where $\Pi'_{n-k}\Gamma_+ = \{x'\in \R^{n-k}:(x_1,\ldots,x_k,x')\in \Gamma_+$ for some $(x_1,\ldots,x_k)\in\R^k\}$), in order to have equality, the minimizer $w_{+}$ for $\Gamma_{+}$ in Proposition 
\ref{prop:main} must depend only on the variables $x_1,\ldots,x_k$. Therefore, $v_{+}(x) = w_{+}(x-p)$ for some $p\in R^n$, and by Theorem \ref{thm:BCF}, $w_{+}$ is an explicit ODE solution depending only on one variable and $\Gamma_{+}$ equals  the intersection of $\Gamma$ with a half-space. That is we can write $\Gamma_{+} = H\times \Gamma'$, $\Gamma_{-} = \left(\R^k\backslash H\right)\times \Gamma'$ for a half-space $H\subset \R^k$. This in particular ensures that $T'(U_2)\cap\Pi'_{n-k} \Gamma_{-}$ also has positive measure in $\R^{n-k}$, and so for equality $w_{-}$ depends only on the variables $x_1,\ldots,x_k$. We have therefore reduced to the setting of Theorem \ref{thm:BKP} in $\R^k$, and so for equality, after a rotation, we have $\Gamma_{+} = \{x_1>0\}\cap \Gamma$ and $\Gamma_{-} = \{x_1<0\}\cap \Gamma$. 

\section{Remarks on transport maps on the sphere $\mS^{n-1}$}

To prove the Friedland-Hayman type inequality, Theorem \ref{thm:main}, we used Proposition \ref{prop:main} to obtain a lower bound on the characteristic exponents $\alpha_\pm $ in terms of Gaussian eigenvalues. This then allowed us to apply the Caffarelli contraction theorem, Theorem \ref{thm:Caf}. A possible alternative to this is to work directly on the sphere $\mS^{n-1}$ and prove an analogous version of the Caffarelli contraction theorem on the sphere. In \cite{Mc1}, McCann shows the following: Let $(M,g)$ be a smooth, compact manifold without boundary, with distance function $d(x,y)$, and let $\mu_1$, $\mu_2$ be probability measures on $M$. There exists a unique mapping $T$ transporting $\mu_1$ onto $\mu_2$ minimizing the total cost with respect to the cost function $c(x,y) = d(x,y)^2/2$. The mapping $T(x) = \exp_{x}\left[-\nabla \psi(x)\right]$ for a $c$-concave function $\psi$, where one defines a function $\psi$ to be $c$-concave if $\psi = \left(\psi^c\right)^c$, for
\begin{align*}
\psi^c(y) = \inf_{x\in M} c(x,y) - \psi(x).
\end{align*}
Now let $\mu$ be the uniformly distributed probability measure on a hemisphere 
$\mS^{n-1}_{+}$ in $\mS^{n-1}$, the appropriate multiple of the volume form
of the round metric. Define the measure $\nu$ by
\begin{align*}
\nu = e^{-f}\mu\big|_{W}.
\end{align*}
Here $W$ is a (geodesically) convex subset of $\mS^{n-1}_+$ and $f$ is a convex function on $W$, such that $\nu$ is a probability measure. It is then natural to pose the following questions, for which the tools of Fathi et al. \cite{FGP} seem appropriate:
\\
\\
1) Does there exist a transport map $T(x) = \exp_{x}\left[-\nabla \psi(x)\right]$ from $\mu$ to $\nu$ for a $c$-concave function $\psi$ such that $T$ is Lipschitz on $\mS^{n-1}_{+}$ with Lipschitz constant bounded by $1$?
\\

\noindent
2) If, at almost every point of $\mS^{n-1}_+$, there is a direction in which $T$ is not a strict contraction, then is it true that, after a rotation, $W$ contains the antipodal points $(\pm 1,0,\ldots,0)$ in  
$\mS^{n-1}_+$ and that $f$ is independent of the $x_1$-variable? \\

\noindent
3)  If there is a pair of points $x$ and $y$ for which $d(x,y) = d(T(x), T(y))$,
then does $W$ contain the full geodesic through $x$ and $y$,
all the way to the antipodal points? And does $T$ split in that direction as in 
Question (2)?  Is there an infinitesimal version of this phenomenon at
a single point?

\end{document}